\documentclass[12pt]{amsart}
\usepackage{setspace}
\usepackage{amsmath, amssymb,amsthm, amsfonts, latexsym, mathrsfs}
\usepackage{color}
\usepackage{graphicx}
\usepackage[normalem]{ulem} 

\newtheorem{thm}{Theorem}[section]
\newtheorem*{thm*}{Theorem}
\newtheorem{lem}{Lemma}[section]

\newtheorem*{prop*}{Proposition}

\theoremstyle{remark}
\newtheorem{rem}{Remark}[section]

\newcommand{\R}{\mathbb R}

\numberwithin{equation}{section}
\newcommand{\be}{\begin{equation}}
\newcommand{\ee}{\end{equation}}
\def\p{\partial}

\def\lf{\left}
\def\ri{\right}
\def\Pi{\displaystyle{\mathbb{II}}}

\def\e{\epsilon}

\def\bX{\mathbf{X}}

\def\wt{\widetilde}

\def\a{\alpha}

\def\ol{\overline}

\def\bee{\begin{equation*}}
\def\eee{\end{equation*}}

\def\H{\mathbb{H}}

\def\exp{\mathrm{exp}}

\newcounter{mnotecount}[section]

\begin{document}
\title[]
{Positivity of Brown-York mass with quasi-positive boundary data}

 \author{Yuguang Shi$^1$}
\address[Yuguang Shi]{Key Laboratory of Pure and Applied Mathematics, School of Mathematical Sciences, Peking University, Beijing, 100871, P.\ R.\ China}
\email{ygshi@math.pku.edu.cn}
\thanks{$^1$Research partially supported by NSFC 11671015 and 11731001  }

\author{Luen-Fai Tam$^2$}
\address[Luen-Fai Tam]{The Institute of Mathematical Sciences and Department of
 Mathematics, The Chinese University of Hong Kong, Shatin, Hong Kong, China.}
 \email{lftam@math.cuhk.edu.hk}
\thanks{$^2$Research partially supported by Hong Kong RGC General Research Fund \#CUHK 14301517}

\renewcommand{\subjclassname}{
  \textup{2010} Mathematics Subject Classification}

\date{February, 2019; revised in \today}

\begin{abstract}
In this short note, we prove positivity of Brown-York mass under quasi-positive boundary data which generalize some previous results by the authors. The corresponding rigidity result is obtained.
\end{abstract}


\maketitle

\markboth{Yuguang Shi and Luen-Fai Tam}{Positivity of Brown-York mass}

\section{Introduction}

Let $(\Omega^n,g)$ be a compact manifold with smooth boundary $\p\Omega$. In this work, we always assume that $\Omega$ is connected and orientable. {\it It is an interesting question to understand the relation between  the geometry of $\Omega$ in terms of scalar curvature  and the intrinsic and extrinsic  geometry of $\p\Omega$ in terms of the mean curvature.}  The question is closely related to the notion of quasi-local mass in general relativity. On other hand, given an  compact manifold $(\Sigma,\gamma)$ without boundary and given a smooth function $H$ on $\Sigma$, one   basic problem  in Riemannian geometry is to study: {\it under what kind of conditions so that  $\gamma$ is  induced  by a Riemannian metric $g$ with  nonnegative  scalar curvature, for example, defined on $\Omega^n$, and $H$ is the mean curvature of $\Sigma$ in $(\Omega^n, g)$ with respect to the outward unit normal vector? } These two problems are closely related and there are no satisfactory answers yet.

 In this kind of study, a result was proved    by the authors which implies the positivity of Brown-York quasi-local mass \cite{BrownYork-1,BrownYork-2},  denoted by $\mathfrak{m}_{BY}(\Sigma;\Omega,g)$. For its definition please see (\ref{e-BY}) below.
More specifically, using the quasi-spherical metrics introduced by Bartnik \cite{Bartnik},   in \cite{ST2} the authors proved the following:

 \begin{thm}\label{postiveBYmass}
	Let $(\Omega^3, g)$ be a compact, connected Riemannian manifold with nonnegative scalar curvature, and with compact mean-convex boundary $\p\Omega$, which consists of spheres with positive Gaussian curvature. Then,
	\begin{equation} \label{e-shi-tam}
		\mathfrak{m}_{BY}(\Sigma_\ell;\Omega,g) \geq 0
	\end{equation}
	for each component $\Sigma_\ell \subset \p\Omega$, $\ell = 1, \ldots, k$. Moreover, equality holds for some $\ell = 1, \ldots, k$ if and only if $\p \Omega$ has only one component and $(\Omega, g)$ is isometric to a domain in $\R^3$.
\end{thm}

Clearly Theorem \ref{postiveBYmass} provides a necessary condition for a boundary data $(\Sigma, \gamma, H)$ to be the one  induced by a Riemannian metric defined on the ambient manifold and with nonnegative scalar curvature. Here $\gamma$ is a metric on $\Sigma$ with quasi positive Gaussian curvature. The existence of qausi-spherical metric in the proof of the theorem  uses   the fact that the mean curvature is {\it positive} at the boundary. Otherwise, it is unclear if one can construct such kind of metrics, see \cite{Bartnik,ST2004}. With these facts in mind, it is natural to ask if Theorem \ref{postiveBYmass} is still true in a more general context. In this note, we consider the problem in the situation of quasi-positive boundary data. Here a function defined on a set is said to be {\it quasi positive} if it is nonnegative and is positive somewhere. The specific results are the following:

\begin{thm} \label{main-quasipositive}
Let $(\Omega, g)$ be a compact three manifold with smooth boundary $\p\Omega$. Let $\Sigma$ be a component of $\p\Omega$. Assume the following:

\begin{enumerate}
 \item [(a)] $\p\Omega$ has nonnegative mean curvature.
 \item [(b)] $\Sigma$ has quasi positive Gaussian curvature.
 \item [(c)] $(\Omega, g)$ has nonnegative scalar curvature.
 \end{enumerate}
 Then we have:
 \begin{enumerate}
 \item[(i)] \underline{Positivity}:
$
\mathfrak{m}_{BY}(\Sigma;\Omega,g)\ge 0.
$

\item[(ii)] \underline{Rigidity}:
Suppose $\mathfrak{m}_{BY}(\Sigma;\Omega,g)=0$, then $\p\Omega$ is connected, $\Omega$ is  homeomorphic to the unit ball in $\R^3$ and $(\Omega,g)$  is a domain in $\R^3$.
\end{enumerate}

\end{thm}

We first remark that in case $\p\Omega$ has quasi positive Gaussian curvature and has {\it positive} mean curvature or $\p\Omega$ has {\it positive} Gaussian curvature and has nonnegative mean curvature, then  the nonnegativity  part  of Theorem \ref{main-quasipositive}  was proved in \cite{ST2004} and \cite{ST3} respectively. However, the rigidity part in the first instance was studied in \cite{ST2004} but not solved very satisfactorily.  The rigidity part in the second instance was not addressed in \cite{ST3}.

To show Theorem \ref{postiveBYmass} we used the method of quasi-spherical metric introduced by Bartnik \cite{Bartnik}. However, if the mean curvature is only assumed to be nonnegative,   a  parabolic equation involved in the quasi-spherical metric may be degenerated. To overcome this difficult, in case $\p\Omega$ is disconnected, we adopt a careful conformal  perturbation on the ambient metric $g$   so that one can use   Theorem \ref{postiveBYmass} and its generalization to the case that the  boundary has positive mean curvature and quasi-positive Gaussian curvature \cite{ST2004}.   In case $\p\Omega=\Sigma$, we use an approximation so that the mean curvature is positive but the scalar curvature may be bounded by a small negative constant. We then embed  the boundary to an hyperbolic space with   negative constant curvature which is small, and use a result in \cite{ShiTam2007} to get nonnegativity of Brown-York mass.

We prove the rigidity part of Theorem \ref{main-quasipositive}, first we show that if the Brown-York mass is zero, then   $\Omega$ is homeomorphic to the unit ball in $\R^3$ and $g$ is scalar flat. Then we need to show that $g$ is Ricci flat.  By suitable approximations,  as in \cite{HI} , one can construct a weak solution of the inverse mean curvature flow (IMCF) in $(\Omega, g)$ with a point $p \in \Omega$ as the initial data (see Lemma \ref{imcf1} below). We then approximate $g$ by metrics so that $\Sigma$ has positive Gaussian curvature and positive mean curvature, and so that it also has zero scalar curvature {\it outside} certain level sets of the IMCF. We can show that the level sets near $p$ have zero Hawking mass. Using the method as in the work of Husiken-Ilmanen \cite{HI}, one then conclude that $g$ is Ricci flat near $p$.

It is still an open question whether the Brown-York mass is nonnegative if the mean curvature is {\it negative} somewhere.

The remaining part of the paper goes as follows:  in the section 2, we prove
the positivity result Theorem \ref{main-quasipositive}; in the section 3, we prove the rigidity result of the theorem.

 {\it Acknowledgment}:  The authors would like to thank Man-Chuen Cheng for many useful discussions.

\section{Positivity}\label{s-positivity}

Let us first clarify the definition of Brown-York mass. Let   $(\Omega, g)$ be compact three manifold with smooth boundary $\p\Omega$. Let $\Sigma$ be a connected component  of $\p\Omega$ with induced metric $\gamma$. Suppose the Gaussian curvature of $(\Sigma,\gamma)$ is quasi positive. Then it can be $C^{1,1}$ isometrically embedded in $\R^3$ as a convex surface with mean curvature $H_0$ which is defined almost everywhere in $\Sigma$. Moreover,
$$
\int_{\Sigma}H_0d\sigma
$$
is well-defined   and is positive, see \cite{GL,HZ,ST2004}. It is well-defined in the sense that it is the same for any $C^{1,1}$ isometric embedding.  Here and below mean curvature is computed with respect to the unit outward normal and the mean curvature of the boundary of the unit ball in $\R^3$ is 2. Hence one can define  the Brown-York mass \cite{BrownYork-1,BrownYork-2} of $\Sigma$ in $(\Omega,g)$ by
\be\label{e-BY}
\mathfrak{m}_{BY}(\Sigma;\Omega,g)=\frac1{8\pi}\int_{\Sigma}(H_0-H) d\sigma.
\ee
Here $H$ is the mean curvature of $\Sigma$ in $(\Omega,g)$. In this section, we want to prove on the  positivity of Brown-York mass in Theorem \ref{main-quasipositive}.

\begin{rem}\label{r-conformal} We always use the following fact. Suppose the scalar curvature $R$ of $(\Omega,g)$ is nonnegative. Let $u$ be the solution of
\bee
\left\{
  \begin{array}{ll}
    8\Delta_g u-Ru=0  \ \ \text{in $\Omega$}  \\
    u=1\ \ \text{on $\p\Omega$.}
  \end{array}
\right.
\eee
Then $u$ is positive, so that $u^4g$ has zero scalar curvature and the mean curvature of $\p\Omega$ with respect to $u^4g$ is no less than its mean curvature with respect to $g$.
\end{rem}

\begin{lem}\label{l-positivity-r1}
Let $(\Omega, g)$ be a compact three manifold with smooth boundary $\p\Omega$ and with nonnegative scalar curvature. Let $\Sigma$ be a component of $\p\Omega$ as in Theorem \ref{main-quasipositive}. Suppose $\p\Omega\setminus \Sigma\neq\emptyset$, then
\bee
\mathfrak{m}_{BY}(\Sigma;\Omega,g)> 0.
\eee

\end{lem}
\begin{proof} In the following, the area element of $\p\Omega$ with respect to the metric induced by $g$ will be denoted by $d\sigma_g$, and the mean curvature will be denoted by $H_g$, etc. Let $\gamma=g|_{T(\Sigma)}$ and let $H_0$ be the mean curvature when $(\Sigma,\gamma)$ is $C^{1,1}$ isometrically embedded in $\R^3$

By Remark \ref{r-conformal}, we may assume that the scalar curvature of $(\Omega,g)$ is zero. Moreover, since $\int_\Sigma H_0d\sigma_g>0$, we may assume that $H(x_0)>0$ somewhere. Let $\Sigma'=\p\Omega\setminus \Sigma\neq\emptyset$.

 First, we want to find a smooth metric $g_1$ on $\ol\Omega$ such that
\begin{enumerate}
\item[(i)] $g_1$ has zero scalar curvature;
  \item [(ii)] the mean curvature $H_{g_1}$ of $\p\Omega$ is positive; and

  \item [(iii)] $g$ and $g_1$ induce the same metric on $\Sigma'$.
\end{enumerate}
To construct $g_1$, let $U$ be a  neighborhood of $x_0$ in $\Sigma$ such that $H_g\ge c_0>0$ in $U$. Let $0\le \phi\le 1$ be a smooth cutoff function with support in $U$ so that $\phi=1$ in a neighborhood of $x_0$.   Given $\e>0$ and let $u $ be the solution of
\bee
\left\{
\begin{array}{rcl}
 \Delta_g u  &  =   &  0  \ \ \mathrm{in} \  \Omega \\
u  &  =  & 1-\e\phi \  \ \mathrm{on} \ \p\Omega.
\end{array}
\right.
\eee
For $\e>0$ small enough, $u>0$ and  has zero scalar curvature. Moreover,
$$
H_{g_1}=\frac1{u^2}\lf(H_g+\frac 4u\frac{\p u}{\p\nu}\ri)
$$
where $\nu$ is the unit outward normal. By the strong   maximum principle $H_{g_1}>0$ outside $U$. Insider $U$, $H_g>0$ and so $H_{g_1}>0$ provided $\e$ is small enough.
 Fix such an $\e_1>0$. Note that the Gaussian curvature of $\Sigma$ may be negative somewhere.  Hence $g_1=u^4g$ satisfies the conditions mentioned above. In particular, the mean curvature at $\Sigma'$ with respect to $g_1$ is bounded below by some positive constant $a>0$.

Next, for any $\e>0$ let $v$ be the harmonic function in $\Omega$ so that $v=1$ on $\Sigma$ and $v=1-\e$ on $\Sigma'$. Then for $\e$ small enough, $v^4g$ is a smooth metric on $\ol\Omega$ such that the mean curvature of $\Sigma$ with respect to $v^4g$ is larger than the mean curvature with respect to $g$. Moreover, the mean curvature of $\Sigma'$ with respect to $v^4g$ is bounded in absolute value by $\frac a2$, provided $\e$ is small enough. Choose such an $\e_2>0$. Let $g_2=v^4g$. Then $g_2, g$ induce the same metric on $\Sigma$ and $(1-\e_2)^4g_1$ and $g_2$ induce the same metric on $\Sigma'$.

Let $M_1= \Omega $ with metric $(1-\e_2)^4g_1$ and $M_2=\Omega$ with metric $g_2$. We can glue the $M_1$ and $M_2$ along $\Sigma'$. Denote the resulting manifold by $M_3$ and  the resulting metric  by $g_3$. Then the boundary of $M_3$ consists of two copies of $\Sigma$ denoted by $\Sigma_1$ and $\Sigma_2$. Moreover the following are true:

\begin{enumerate}
  \item [(i)] $g_3$ is smooth except along $\Sigma'$. Moreover, $g_3$ is Lipschitz and is smooth on each side of $\Sigma'$.
  \item [(ii)] The scalar curvature of $g_3$ is zero away from $\Sigma'$.
  \item [(iii)] The mean curvature of $\Sigma_1$ and $\Sigma_2$ are positive.
  \item[(iv)] The mean curvature jump at $\Sigma'$ is positive. Namely, if we choose the unit normal pointing outside $\Sigma'$ in $M_1$, then the mean curvature jump is at least $a-\frac a2=\frac a2>0$.
      \item[(v)] $g$ and $g_3$ induce the same metric on $\Sigma$ which corresponds to $\Sigma_2$.
          \item[(vi)] The mean curvature of $\Sigma_2$ with respect to $g_3$ is larger than the mean curvature of $\Sigma$ with respect to $g$.
\end{enumerate}

We claim that
\be\label{e-positive-r-1}
\int_{\Sigma_2}(H_0-H_{g_3})d\sigma_{g_3}\ge0.
\ee
  If the claim is true, then by (v) and (vi) above, we conclude the lemma is true.

To prove the claim we further glue $M_3$ along $\Sigma_1$. Denote the resulting manifold by $M_4$ and  the resulting metric  by $g_4$.  The boundary of $M_4$ consists of two copies of $\Sigma_2$, denoted by $\wt\Sigma_1, \wt\Sigma_2$. The following are true:
\begin{enumerate}
  \item [(i)] $g_4$ is smooth except along those parts coming  from $\Sigma'$ or from $\Sigma_1$. Moreover, $g_4$ is Lipschitz and is smooth on each side of these surfaces.
  \item [(ii)] The scalar curvature of $g_4$ is zero away from those parts coming  from $\Sigma'$ or from $\Sigma_1$.
  \item [(iii)] The mean curvature of $\wt\Sigma_1$ and $\wt\Sigma_2$ with respect to $g_4$ are positive. In fact they are equal the mean curvature of $\Sigma_2$ with respect to $g_3$.
  \item[(iv)] The mean curvature jump at those parts coming  $\Sigma'$ or $\Sigma_1$ are positive, because the  mean curvature of $\Sigma_1$ with respect to $g_3$ is positive.
      \item[(v)]  $\wt\Sigma_1, \wt \Sigma_2$ with respect to the induced metric from $g_4$ is isometric to $(\Sigma,g|_{T(\Sigma)})$.
\end{enumerate}

By \cite[Theorem 3.3]{M-M-T}, there exists a smooth metric $h$ on $M_4$ with nonnegative scalar curvature so that $h$,   and $g_4$ induce the same metric on $\p M_4$ and
$$
\int_{\p M_4}H_h d\sigma_h>\int_{\p M_4} H_{g_4}d\sigma_{g_4}=2\int_\Sigma H_{g_3}d\sigma_{g_3}.
$$
  Moreover, $H_h>0$ on $\p M_4$.  Since each component of $\p M_4$ with metric induced by $h$ is isometric to $ \Sigma$ with metric induced by $g$, it has quasi positive Gaussian curvature. By \cite[Theorem 0.2]{ST2004}, we conclude that
$$
2\int_{\Sigma} H_0 d\sigma\ge  \int_{\p M_4}H_h d\sigma\ge 2\int_\Sigma H_{g_3}d\sigma_{g_3}.
$$
Hence the claim is true. This completes the proof of the lemma.

\end{proof}

\begin{lem}\label{l-positivity-r2} Let $(\Omega,g)$ and $\Sigma$ be as in Theorem \ref{main-quasipositive}. Suppose $\p\Omega=\Sigma$, then
\bee
\mathfrak{m}_{BY}(\Sigma;\Omega,g)\ge 0.
\eee
\end{lem}

\begin{proof} By Remark \ref{r-conformal}, we may assume that $g$ is scalar flat. Note that $\p\Omega=\Sigma$  is a sphere because its Gaussian curvature is quasi positive. Moreover, we may assume   the mean curvature $H$ of $\Sigma$ is quasi positive.  Let $x_0\in \Sigma$ with $H(x_0)>0$. Let $U$ be an neighborhood of $x_0$ in $\Sigma$ such that $H_g\ge c_0>0$ in $U$. Let $0\le \phi\le 1$ be a smooth cutoff function with support in $U$ so that $\phi=1$ in a neighborhood of $x_0$.   Given $\e>0$ and let $u=u(\e) $ be the solution of
\bee
\left\{
\begin{array}{rcl}
 \Delta_g u  &  =   &  0  \ \ \mathrm{in} \  \Omega \\
u  &  =  & 1-\e\phi \  \ \mathrm{on} \ \p\Omega.
\end{array}
\right.
\eee
For $\e>0$ small enough, $  g(\e)=u^4g$ has zero scalar curvature so that $\p\Omega$ has positive mean curvature. Let $ \gamma(\e)$ be the metric on $\Sigma$ induced by $ g(\e)$ and let $  K(\e)$ be the Gaussian curvature of $\Sigma$ with respect to $  \gamma(\e)$. Then

\be\label{e-Gauss-1}
  K(\e)> -\kappa^2(\e)
\ee
where $\kappa(\e)>0$, $\kappa(\e)\to0$ as $\e\to 0$.     We can isometrically embed $(\Sigma, \gamma(\e))$ in $\H_{-\kappa^2(\e) }$ as a strictly convex surface in the ball model defined in  the ball
$$
\{|x|<\kappa^{-2}(\e)\}
$$
by \cite{Pogorelov}.
Moreover, we may assume the origin is inside the embedded surface. Let $  H(\e)$ be the mean curvature of $\Sigma$ with respect to $  g(\e)$ and let $H_{\kappa(\e)}$ be the mean curvature when $(\Sigma, \gamma(\e))$ is isometrically  embedded in the hyperbolic space $\H_{-\kappa^2(\e)}$ with constant curvature $-\kappa(\e)$. By \cite{ShiTam2007}, we have
\be\label{e-AH}
\int_\Sigma(H_{\kappa(\e)}-H(\e))\cosh (\kappa(\e) r) d \sigma_{g(\e)} \ge0
\ee
where $r$ is the distance from the origin in $\H_\kappa(\e)$.

Observe that we can find $\e_i\to0$ such that $g(\e_i)\to g$ in $C^\infty$ norm on $\ol\Omega$. Hence the
  intrinsic diameter   of $(\Sigma,\gamma(\e_i))$ is bounded by a constant independent of $i$, we conclude that $r$ is bounded by a constant independent of $i$.
  By \cite[p.7152-7154]{LinWang2015}, one can choose $\e_i\to0$ such that:

\begin{itemize}
  \item $H_{\kappa(\e_i)}$ are uniformly bounded from above. (Note that $H_{\kappa(\e_i)}>0$).
  \item If $\bX_i=(x^1,x^2,x^3)$ is the isometric embedding of $(\Sigma,\gamma(\e_i))$, then the $C^2$ norm with respect to the fixed metric $\sigma$  are uniformly bounded.
\end{itemize}
Together with \eqref{e-AH}, we conclude that
$$
\liminf_{i\to\infty}\int_\Sigma(H_{\kappa(\e_i)}-H_g )  d\sigma\ge0.
$$
Moreover, $\bX_i$ converge to a $C^{1,1}$ embedding of $(\Sigma, \sigma)$ in $\R^3$ as a convex surface. As in \cite{ST2004}, one can  conclude that
$$
\lim_{i\to\infty}\int_\Sigma H_{\kappa(\e_i)}   d\sigma=\int_\Sigma H_0d\sigma.
$$
 where $H_0$ is the mean curvature of $ \Sigma $ when $(\Sigma,\gamma)$   is isometrically $C^{1,1}$ embedded in $\R^3$. Here $\gamma=g|_{T(\Sigma)}$.  From this the lemma follows.

 \end{proof}

 \begin{proof}[Proof of   Theorem \ref{main-quasipositive} (i)  \underline{Positivity}] Let $(\Omega,g)$, $\Sigma$ be as in Theorem  \ref{main-quasipositive}. Then by Lemmas \ref{l-positivity-r1} and \ref{l-positivity-r2}, we have
 \bee
\mathfrak{m}_{BY}(\Sigma;\Omega,g)\ge 0.
\eee

 \end{proof}

\section{Rigidity}\label{s-rigidity}

In the section, we will prove the rigidity part in Theorem \ref{main-quasipositive}. First we have the following:
\begin{lem}\label{l-positivity-r3}
Let $(\Omega,g), \Sigma$ be as in Theorem \ref{main-quasipositive} so that $\p\Omega=\Sigma$. Suppose $\Omega$ is not homeomorphic to the unit ball in $\R^3$, then
\bee
\mathfrak{m}_{BY}(\Sigma;\Omega,g)> 0.
\eee
\end{lem}

\begin{proof} Since the Gaussian curvature of $\Sigma$ is quasi positive, $\Sigma$ is a topological sphere. If $\Omega$ is a handle body, then it is homeomorphic to the unit ball. Suppose this is not the case, then $\Omega$ is not a handle body. By \cite[Theorem 1' and Proposition 1]{MeeksSimonYau} there is an embedded minimal surface $S$ which is either a sphere or a minimal projective space inside $\Omega$.

{\bf Case 1}: Suppose $S$ is a sphere. Since $S$ is orientable, there is a smooth unit normal vector field on $S$ and there is an embedding $F:S\times(-1,1)\to \Omega$ so that $F(\cdot\, , 0)=S$ and the image of $F$ is a tabular neighborhood $N$ of $S$ in $\Omega$. Then
$N\setminus S$ is a manifold with boundary which are two copies of $S$ with two components. Hence $\Omega\setminus S$ is a manifold  with boundary which is a copy of $S$. Let $\wt\Omega$ be the connected component containing $\p\Omega=\Sigma$ of this manifold. Then $(\wt\Omega,g)$ has nonnegative scalar curvature so that $\p\wt\Omega$ is disconnected, and
$ \mathfrak{m}_{BY}(\Sigma,\Omega,g)=\mathfrak{m}_{BY}(\Sigma,\wt\Omega,g)$, which is   positive by Lemma \ref{l-positivity-r1}.

{\bf Case 2}: Suppose $S$ is a projective space. $f:\mathbb{RP}^2\to \Omega$ is an embedding with $S=f(\mathbb{RP}^2)$. We want to construct a double cover $p:\hat \Omega\to \Omega$ so that $p^{-1}(f(\mathbb{RP}^2))\cong\mathbb{S}^2$.

Let $V$ be the normal bundle of the embedding $f$. Note that $\mathbb{RP}^2$ has only two non-isomorphic real line bundles, namely the tautological line bundle and the trivial one. Since $\Omega$ is orientable, $V$ is isomorphic the tautological line bundle $((\mathbb{S}^2\times\mathbb{R})/\sim)\to(\mathbb{S}^2/\sim)\cong\mathbb{RP}^2$ with $(x,k)\sim(-x,-k)$ on $\mathbb{S}^2\times\mathbb{R}$.

By the tubular neighborhood theorem, there exists an open embedding $G:((\mathbb{S}^2\times\mathbb{R})/\sim)\cong V\to \Omega$ whose restriction on the zero section is equal to $f$. Let $\Omega'=G((\mathbb{S}^2\times[-1,1])/\sim)$ and $\Omega''=\Omega\setminus G((\mathbb{S}^2\times(-1,1))/\sim)$.  Then $\Omega=\Omega'\cup \Omega''$ with $\Omega'\cap \Omega''=\partial \Omega'\cong\mathbb{S}^2$.

Let $\Omega_+$, $\Omega_{-}$ be two identical copies of $\Omega''$. Define $\phi:\mathbb{S}^2\times\{-1,1\}\to \Omega_+\sqcup \Omega_-$ by $\phi(x,1)=g([(x,1)])\in \Omega_+$ and $\phi(x,-1)=g([(x,-1)])\in \Omega_-$ for $x\in\mathbb{S}^2$. Let $\hat \Omega=\mathbb{S}^2\times[-1,1]\cup_{\phi}(\Omega_+\sqcup \Omega_-)$. Then the obvious map $p:\hat \Omega\to \Omega$ has the desired properties. By the construction, we see that $(\hat \Omega, \hat g)$ has nonnegative scalar curvature and $\p \hat \Omega$ two components, each of them has quasi-positive mean curvature  with respect to outward unit norm vector and  quasi-positive Gauss curvature. In fact, near each component, $(\hat\Omega,\hat g)$ is isometric to neighborhood of $\Sigma$ in $(\Omega,g) $.  On the other hand, $ 2\mathfrak{m}_{BY}(\Sigma,\Omega,g)=\mathfrak{m}_{BY}(\p\hat \Omega,\hat\Omega,g)$, which is positive by Lemma \ref{l-positivity-r1}. This completes the proof of the lemma.
\end{proof}

Let $(\Omega,g)$ and $\Sigma$ be as in Theorem \ref{main-quasipositive}. Suppose $\mathfrak{m}_{BY}(\Sigma;\Omega,g)=0$. Then by Lemmas \ref{l-positivity-r1} and \ref{l-positivity-r3}, we conclude that $\p\Omega=\Sigma$ and $\Omega$ is homeomorphic to the unit ball. By Remark \ref{r-conformal}, we conclude that $g$ is scalar flat.  Moreover, since $\Sigma$ has quasi positive Gaussian curvature, we conclude that $\Sigma$   has quasi positive mean curvature. In the rest of this section, we always assume the above facts. In remains to prove that $g$ is Ricci flat.

  We need the following two lemmas.

 \begin{lem}\label{l-approx-1}  Let $(\Omega, g)$ and $\Sigma$ be as above.   For any $p$ in $\Omega$ and for any $\rho>0$ small enough,  there is a sequence of smooth metrics $g_i$ on $\ol\Omega$ with the following properties:

\begin{enumerate}
  \item [(i)] $g_i\to g$ in $C^\infty$ norm in $\ol \Omega$.
  \item [(ii)] $\Sigma$ has positive mean curvature $H_i$ with respect to $g_i$.
  \item [(iii)] Let $\gamma_i$ be the induced metric of $g_i$ on $\Sigma$. Then the Gaussian curvature of $(\Sigma,\gamma_i)$ has positive Gaussian curvature.
      \item[(iv)] The scalar curvature of $g_i$ is zero outside $B(p,2\rho)$.
      \item[(v)] The mean curvature of $\p B_g(p,s)$ with respect to $g_i$ is positive for all $s<2\rho$ for all $i$.
          \item[(vi)] $
\mathfrak{m}_{BY}(\Sigma;\Omega,g_i)\to 0$ as $i\to\infty$.
\end{enumerate}

\end{lem}
\begin{proof} Let $\rho>0$ be small enough so that $\p B_g(p,s)$ is diffeomorphic to the sphere so that its mean curvature is larger than $1/s$ for all $0<s<2\rho$. Fix a smooth cutoff function $\phi\ge0$ so that $\phi=1$ in $B(p,\rho)$ and $\phi=0$ outside $B(p,2\rho)$. Let $v$ be the solution of $\Delta_g v=\e \phi$ in $\Omega$ and $v=1$ on $\Sigma$. Then for $\e>0$ small enough, $v>0$. Let $g_\e =v^4 g$. For $\e$ small enough, $g_\e$ satisfies:

\begin{itemize}
  \item $g_\e\to g$ in $C^\infty$ norm in $\ol \Omega$.
  \item The scalar curvature of $g_\e$ is zero outside $B(p,2\rho)$.
  \item The mean curvature of $\Sigma$ with respect to $g_\e$ is positive. This follows by strong maximum principle that $\frac{\p v}{\p \nu}>0$ where $\nu$ is the unit outward normal of $\Sigma$ with respect to $g$.
\end{itemize}
Since $v=1$ on $\Sigma$, the metrics induced by $g, g_\e$ are equal, and will be denoted by $\gamma$. In particular,  the Gaussian curvature of $\Sigma$ does not change. If the Gaussian curvature of $(\Sigma,\gamma)$ is positive, then $g_{\e}$ are the required metrics.
  Otherwise, we can find a smooth function $\eta$ on $\Sigma$ such that $\eta \le0$, $\Delta_\gamma\eta=-1$ in an open set containing $\{K=0\}$. For fixed $\e>0$, for $\tau>0$, and let $w$ be the solution of $\Delta_{g_\e} w=0$ in $\Omega$ so that $w=\exp(\frac12\tau \eta)$. Let $h_\tau= w^4g_\e$. Then
\begin{itemize}
  \item $h_\tau \to g_\e$ in $C^\infty$ norm in $\ol \Omega$ as $\tau \rightarrow 0$.
  \item The scalar curvature of $h_\tau$ is zero outside $B(p,2\rho)$.
  \item The mean curvature of $\Sigma$ is positive, provided $\tau$ is small enough.
  \item The Gaussian curvature of $\Sigma$ with respect to the metric induced by $h_\tau$ is positive provided $\tau$ is small enough.
\end{itemize}

From these, it is easy to see the lemma is true.

\end{proof}

 The following lemma is basically from \cite{HI}.
\begin{lem}\label{imcf1} Let $(\Omega,g)$, $\Sigma$ be as above.
 For any $p \in  \Omega $, there is a weak solution for the inverse mean curvature flow in $(\Omega, g)$ with $p$ as the initial data.
\end{lem}

\begin{proof}

Let $U$ be a small neighborhood of $\p\Omega$, then extend $\Omega \cup U $ to be Euclidean near infinity, the resulting metric is denoted by $\hat g$.

Let us consider the inverse mean curvature flow (IMCF) in $(M,\hat g)$ with $\partial B_r(p)$ as the initial data where $r>0$ is small enough. By Theorem 3.1 in \cite{HI}, there is a weak solution $u_{r}$ to this IMCF with $u_{r}|_{\partial B_r(p)}=0$ and

$$
|\nabla u_{r}|(x)\leq \sup_{\p B_r (p) \cap B_{\rho}(x) } H_+ +\frac{C}{\rho},
$$
for any $0<\rho\leq \sigma(x)$, here  $C$ is a universal constant independent on $\rho$ and $r$, $\sigma(x)$ is defined in Definition 3.3 in \cite{HI}, i.e. for any $x\in \Omega$, let  $\tau(x)\in (0, \infty]$  be the supremum of radii $r$ such that $B_r (x)\subset \Omega$, and

$$
Rc\geq-\frac1{1000r^2} ~\text{in} ~ B_r (x),
$$
and there is a $C^2$ function $p$ on $B_r (x)$ such that $p(x)=0$, $p\geq d^2(, x)$, and $|\nabla p|\leq 3 d(, x)$, $\nabla^2 p \leq 3g$ on $B_r (x)$, define $\sigma(x)=\min\{\tau(x), d(x, \p\Omega)\}$.  Let $\Omega'\subset \subset \Omega$ with $dist(\p \Omega', \p\Omega )$ being any fixed small number and $p\in \Omega'$. Without loss of the generality, it suffices to consider the case that $x\in \Omega' $, so, we may assume $\sigma(x)\geq \sigma_0$ for any $x\in \Omega' $, here $\sigma_0$ is a fixed number  depends only  on $dist(\p \Omega', \p\Omega )$ and $(\Omega ,g)$.

Let us choose $r$ small enough so that $\sup_{\p B_r (p) } H_+ \leq \frac 3r$. Now, we claim that for any $x\in \Omega' $

 \begin{equation}\label{estimate1}
|\nabla u_{r}|(x)\leq \frac{C}{d(x,p)},
\end{equation}
here $C $ is a  universal constant independent on $r$, $d(x,p )$ is the distance function to $p$  with respect to the metric $g$.

In fact, if $d(x,p)\leq 4r$, then we take $\rho =\frac{r}2$, here we assume $r\leq \frac{\sigma_0}2$, we get (\ref{estimate1}); if $d(x,p)> 4r$, let $\rho=\min \{\frac12 dist(x,p), \frac{\sigma_0}2\}$, together with the fact $dist(x,p)\leq \Lambda \sigma_0$, where $\Lambda$ is a universal constant, we still get (\ref{estimate1}).

On the other hand,  together with Theorem 2.1 in \cite{HI} and the remarks following it, we know that by taking a subsequence of $\{ u_{r}\}$, denoted by $\{ u_{r_i}\}$, there is a constant $C_i$  so that $\{ u_{r_i}-C_i\}$ converges to the weak solution of IMCF $-\infty <u$  in $(\Omega', g)$ with $p$ as the initial data. Note that the mean curvature of  $\partial B_r(p)$ is positive for all $r\leq \delta$, we see that the level set of $u$ in $ B_\delta (p) \subset \subset\Omega' $ cannot jump, and

$$
|\nabla u|(x)\leq \frac{C}{d(x,p)},
$$

and $-\infty <u\leq t_0$, here $t_0$ is a universal constant.
\end{proof}

Let us first recall the definition of {\it minimizing hull} in $\Omega$.
A subset $E$ of $\Omega$   with locally finite perimeter said to be a minimizing
hull in $\Omega$ if $|\p^*E\cap K|\le |\p^*F\cap K|$ for any  set $F\subset \Omega$ with locally finite perimeter such that $F\supset E$ and  $F \setminus E \Subset\Omega$  and for any compact set $K$ with $F \setminus E \subset K \subset\Omega$.
Here $\p^*E, \p^*F$ are  the reduced boundaries of E and F respectively.

By the proof in  \cite[Theorem 2.5]{ST}, we see that for $t$ small enough,  the slice $N_t=\p\{u<t\}$ of the weak IMCF in Lemma \ref{imcf1} is the boundary of a {\it minimizing hull}  in $(\Omega,g)$ with $C^{1,\a}$ smooth and $\int_{N_t}|A|^2 d\sigma <\infty$, and $\mathfrak{m}_H(N_t)\geq 0$.

We are ready to prove the rigidity part of Theorem \ref{main-quasipositive}.

\begin{proof}[Proof of Theorem \ref{main-quasipositive} (ii) \underline{Rigidity}] Let $p\in \Omega$. Suppose $g$ is not flat near $p$. Choose $r>0$ be small enough with $B(p,2r)\Subset \Omega$,  so that $\p B(p,s)$ is a sphere with mean curvature at least $1/s$ for all $s<2r$.  Then by Lemma \ref{imcf1} and \cite{HI},
  one can find a solution to the IMF given by a locally Lipschitz function $u$, so that for some $a$, the following are true: (i) $E_t=\{u<t\}$ is precompact in  $B(x,r)$ for $t<a$; (ii) $\p E_t$ is connected; (iii) $E_t$ is a minimizing hull in $(\Omega,g)$; (iv) $\mathfrak{m}_H(\p E_t,g)>0$, for $t<a$.

Fix $t_0<a$ so that $\mathfrak{m}_H(\p E_{t_0},g)\ge b$ for some $b>0$. In the following we denote $E_{t_0}$ by $E$. For any $\theta>0$ small enough, we can find   $E\subset F\Subset B(x,r)$  such that
\be
|\p E|_g\le |\p F|_g\le |\p E|_g+\theta; \  \mathfrak{m}_H(\p F) \ge\mathfrak{m}_H(\p E)-\theta>0.
\ee
Moreover $\p F$ is smooth. Note that $F$ depends on $\theta$.

Since $p\in E_{t_0}$ which is open, we can find   $r>\rho>0$ such that   $B(p,2\rho)\Subset E$.

Next, we want to approximate $g$. By the Lemma \ref{l-approx-1} , for any $\e>0$  small enough, we can find a smooth metric $g_\e$ on $\ol\Omega$ so that (i) $||g-g_\e||_{C^4}\le \e$;
  (ii)  $\Sigma$ has positive mean curvature $H_\e$ with respect to $g_\e$; (iii)  The Gaussian curvature of $(\Sigma,g_\e|_{T(\Sigma)}) $ has positive Gaussian curvature.
      (iv) the scalar curvature of $g_\e$ is zero outside $B(p,2\rho)$;
      (v) The mean curvature of $\p B(p,s)$ with respect to $g_\e$ is positive for all $s<2r$;
           (vi)   $\mathfrak{m}_{BY}(\Sigma, \Omega,g_\e)\le \e$; (vii) $
  |\p F|_{g_\e}  \le |\p E|_g+\theta+\e$, $  \mathfrak{m}_H(\p F,g_\e) \ge\mathfrak{m}_H(\p E,g))-\theta-\e>0$.

By (ii), (iii), we can glue $\Omega$ to the exterior of the a convex set in $\R^3$ so that the scalar curvature outside the convex set is zero and is asymptotically flat. Denote the manifold by $M$. We still denote this metric as $g_\e$. Note that $g_\e$ has zero scalar curvature outside $B(x,2r)$. However,  $g_\e$ may have negative scalar curvature inside $B(p,2\rho)$. By the monotonicity in qausi-spherical metric \cite{ST2}, using the Lemma \ref{l-approx-1} (vi) we may choose $g_\e$ so that
$$
\mathfrak{m}_{ADM}(g_\e)\le \e.
$$

Fix such an $\e$.
Using the method of Miao \cite{M}, for $\tau>0$ small enough, we can find metrics $h_\tau$ so that $h_\tau=g_\e$ outside $\{x\in M| d_{g_\e}(x,\Sigma)<\tau\}$ and the scalar curvature inside $\{x\in M| d_{g_\e}(x,\Sigma)<\tau\}$ is uniformly bounded. Let $R_\tau$ be the scalar curvature of $g_\e$. One can find a positive solution of
$$
\wt R_\tau u-8\Delta_{g_\e}u=0
$$
with $u\to 1$ near infinity. Here $\wt R_\tau=R_\tau$ in $\{x\in M| d_{g_\e}(x,\Sigma)<\tau\}$
and $\wt R_\tau=0$ outside this set. Note that $\wt R_\tau$ is smooth. Hence one can   approximate $g_\e$  by a smooth metrics $h_\tau=u^4g_\e$ on the manifold so that, $h_\tau$ has zero scalar curvature outside $B(p,2\rho)$ and
$$
\mathfrak{m}_{ADM}(h_\tau)\le 2\e.
$$
Moreover, $h_\tau\to g_\e$ uniformly in $M$, $h_\tau\to g_\e$ in $C^\infty$ norm in any compact set away from $\Sigma$.

  Note that the mean curvature of $\Sigma$ with respect to  $g_\e$ is positive and $\mathfrak{m}_H(\p F,g_\e)>0$, one can find   $F_\e$ which is  the minimizing hull of $F$ with respect to  $g_\e$ inside $\Omega$. $F_\e$ exists because the mean curvature of $\Sigma=\p\Omega$ is positive with respect to $g_\e$. Then $F_\e\Subset \Omega$ and is connected because  $M$ is homeomorphic to $\R^3$. Using the fact that the scalar curvature of $g_\e$ is zero outside $\p F$, one   can proceed as in the proof \cite[Theorem 3.1]{ST2004}, to obtain

$$
2\e\ge\mathfrak{m}_{ADM}(g_\e)\ge \mathfrak{m}_H(\p F_\e,g_\e).
$$

On the other hand, the mean curvature of $\p F_\e$ is zero on $\p F_\e\setminus \p F$ is equal to the mean curvature of $\p F$ on $\p F_\e\cap \p F$
\bee
\begin{split}
\mathfrak{m}_H(\p F_\e,g_\e)=&\sqrt{\frac{|\p F_\e|_{g_\e}}{16\pi}} \lf(1-\frac1{16\pi}\int_{\p F_\e }H^2 d\sigma_{g_\e}\ri)\\
\ge &\sqrt{ \frac{|\p F_\e|_{g_\e}}{16\pi}} \lf(1-\frac1{16\pi}\int_{\p F }H^2 d\sigma_{g_\e}\ri)\\
=&\sqrt{ \frac{|\p F_\e|_{g_\e}}{|\p F|_{g_\e}}}\mathfrak{m}_H(\p F,g_\e)\\
\ge & \sqrt{ \frac{|\p F_\e|_{g_\e}}{|\p F|_{g_\e}}}(\mathfrak{m}_H(\p E,g)-\theta-\e).
\end{split}
\eee
Now
\bee
\begin{split}
|\p F|_{g_\e}\le  & \lf(|\p E|_g+\theta+\e\ri)\\
\le & \lf(|\p F_\e|_g+\theta+\e\ri)\\
\le & (1+\e)\lf(|\p F_\e|_{g_\e}+\theta+\e\ri)
\end{split}
\eee
and
\bee
|\p F|_{g_\e}\ge (1-\e)|\p F|_g\ge (1-\e)|\p E|_g
\eee
here we may assume that $(1+\e)^{-1}g\le g_\e\le (1+\e)g$. Hence
\bee
\begin{split}
\frac{|\p F_\e|_{g_\e}}{|\p F|_{g_\e}}\ge &\frac1{1+\e}-(\theta+\e)\cdot\frac1{|\p F_\e|_{g_\e}}\\
\ge& \frac1{1+\e}-(\theta+\e)\cdot\frac1{(1-\e)|\p E|_g}
\end{split}
\eee
Since
$(\mathfrak{m}_H(\p E,g)-\theta-\e)>0$ provided $\theta, \e$ are small enough, we have
\bee
2\e\ge \lf(\frac1{1+\e}-(\theta+\e)\cdot\frac1{(1-\e)|\p E|_g}\ri)^\frac12(\mathfrak{m}_H(\p E,g)-\theta-\e).
\eee
Let $\e\to 0$ and then let $\theta\to0$, we have
$$
0\ge \mathfrak{m}_H(\p E,g)>0.
$$
This is a contradiction.

\end{proof}

\begin{rem} It is not difficult to see that by the arguments in the above proof of rigidity, we may also get
$\mathfrak{m}_{BY}(\Sigma;\Omega,g)\geq 0$ in case $\Omega$ is homeomorphic to a ball.
\end{rem}

\end{document}